\documentclass[12pt, reqno, twoside, letterpaper]{amsart}

\usepackage{paperstyle}
\usepackage{graphicx}
\usepackage{float}

\usepackage{mathtools}
\usepackage{todonotes}
\usepackage[norefs,nocites]{refcheck}

\usepackage{bbm}

\newcommand{\be}{\begin{equation}}
\newcommand{\ee}{\end{equation}}
\newcommand{\dalign}[1]{\[\begin{aligned} #1 \end{aligned}\]}
\newcommand{\euB}{\EuScript{B}}
\newcommand{\er}{\mathrm{e}}
\newcommand{\llsym}[1]{\,\mathop{\ll}\limits_{#1}\,}

\newcommand{\gch}{\chi}
\newcommand{\gchm}{q}
\newcommand{\ogch}{\overline\gch}
\newcommand{\RHgch}{{\tt RH}[\gch]}
\newcommand{\RHsimgch}{{\tt RH}_{\tt sim}[\gch]}

\newcommand{\LHgch}{{\tt LH}[\gch]}

\newcommand{\gcho}{\gch_\circ}

\newcommand{\ogcho}{\overline\gcho}

\newcommand{\RHsimgcho}{{\tt RH}_{\tt sim}[\gcho]}

\newcommand{\pch}{{\mathfrak X}}
\newcommand{\pchm}{{\mathfrak q}}
\newcommand{\opch}{\overline\pch}

\newcommand{\RHsimpch}{{\tt RH}_{\tt sim}[\pch]}
\newcommand{\LHpch}{{\tt LH}[\pch]}

\newcommand{\RHdagsimpch}{{\tt RH}_{\tt sim}^\dagger[\pch]}

\newcommand{\pcho}{\pch_\circ}
\newcommand{\pchom}{k}

\newcommand{\RHsimpcho}{{\tt RH}_{\tt sim}[\pcho]}

\newcommand{\one}{{\mathbbm 1}}
\newcommand{\RHone}{{\tt RH}[\one]}
\newcommand{\LHone}{{\tt LH}[\one]}

\newcommand{\RHsimstar}{{\tt RH}_{\tt sim}[\boldsymbol\star]}

\newcommand{\LHstar}{{\tt LH}[\boldsymbol\star]}

\newcommand{\otheta}{\overline\vartheta}
\newcommand{\tD}{{\widetilde D}}
\newcommand{\tL}{{\widetilde L}}
\newcommand{\oR}{{\cR_\bullet}}
\newcommand{\oV}{{\cV_\bullet}}

\title[A variant of the Linnik-Sprind\v zuk theorem]
{A variant of the Linnik-Sprind\v zuk theorem\break
for simple zeros of Dirichlet $L$-functions}
      
\author[W.~Banks]{William D.~Banks}

\address{Department of Mathematics, 
         University of Missouri, 
         Columbia MO 65211, USA.}

\email{bankswd@missouri.edu}
        
\date{\today}

\begin{document}

\begin{abstract}
For a primitive Dirichlet character $\pch$, a new hypothesis $\RHdagsimpch$
is introduced, which asserts that (1) all \emph{simple} zeros of $L(s,\pch)$ 
in the critical strip are located on the critical line,
and (2) these zeros satisfy some specific conditions on their
vertical distribution. We show that $\RHdagsimpch$ (for any $\pch$)
is a consequence of the \emph{generalized Riemann hypothesis}.

Assuming only the \emph{generalized Lindel\"of hypothesis}, we show that if
$\RHdagsimpch$ holds for one primitive character $\pch$, then it holds for \emph{every} such $\pch$.
If this occurs, then for every character~$\gch$ (primitive or not), all simple zeros of $L(s,\chi)$
in the critical strip are located on the critical line. In particular, Siegel zeros cannot exist
in this situation.
\end{abstract}

\thanks{MSC Numbers: Primary: 11M06, 11M26; Secondary: 11M20.}

\thanks{Keywords: generalized Riemann hypothesis, Riemann zeta function,
Lindel\"of hypothesis, simple zeros.}

%\thanks{Data Availability Statement: Data sharing not applicable to this article as no datasets
%were generated or analysed during the current study.}
%
%\thanks{Potential Conflicts of Interest: NONE}
%
%\thanks{Research Involving Human Participants and/or Animals: NONE}

\maketitle

%%%%%%%%%%%%%%%%%%%%%%%%
%%%%% PAPER BEGINS %%%%%
%%%%%%%%%%%%%%%%%%%%%%%%

\tableofcontents

\section{Introduction}

An old result of Sprind\v zuk~\cite{Sprind1,Sprind2}, which he obtained
by developing ideas of Linnik~\cite{Linnik}, asserts that the
\emph{generalized Riemann hypothesis} (GRH) holds for all
Dirichlet $L$-functions provided that the \emph{Riemann hypothesis} (RH)
is true and that certain conditions on the vertical distribution of the zeros
of $\zeta(s)$ are met. Specifically, Sprind\v zuk showed under RH that
every $L$-function $L(s,\chi)$ satisfies GRH
provided that the asymptotic formula
\[
\sum_{\gamma}|\gamma|^{i\gamma}\er^{-i\gamma-\pi|\gamma|/2}
\Bigl(x+2\pi i\frac{h}{k}\Bigl)^{-1/2-i\gamma}
=-\frac{\mu(k)}{x\sqrt{2\pi}\,\phi(k)}+O(x^{-1/2-\eps})
\]
holds as $x\to 0^+$ for any coprime integers $h,k$ with
$0<|h|\le k/2$, where the sum runs over the ordinates $\gamma$
of the nontrivial zeros of $\zeta(s)$. This is known as the
\emph{Linnik–Sprind\v zuk theorem}.
Some similar results have
been obtained by Fujii~\cite{Fujii1,Fujii2,Fujii3}, Suzuki~\cite{Suzuki},
Kaczorowski and Perelli~\cite{KacPer}, and the author~\cite{Banks1,Banks2,Banks3}.

In the present paper, we establish a variant
of the Linnik–Sprind\v zuk theorem focused on
the \emph{simple} zeros of Dirichlet \text{$L$-functions}.
Assuming the \emph{generalized Lindel\"of hypothesis},
we show that the horizontal and vertical distribution of the simple zeros
of any single \text{$L$-function} $L(s,\gch)$ can strongly influence
the horizontal and vertical
distribution of the simple zeros of any other Dirichlet $L$-function.
In particular, we give a new criterion for the \emph{nonexistence of 
Siegel zeros}.

Our results are formulated in \S\ref{sec:notation} after
the necessary notation has been introduced.
To prove the main theorem, we study the interplay
and consequences of various hypotheses on the $L$-functions attached
to Dirichlet characters $\gch$.

\bigskip\noindent{\sc Hypothesis $\RHgch$}:
{\it If $\rho=\beta+i\gamma$ is a zero of $L(s,\gch)$ with $\beta>0$, 
then $\beta=\frac12$.}

\bigskip\noindent In particular, $\RHone$ is the Riemann hypothesis, where $\one$ is the trivial
character defined by $\one(n)\defeq 1$ for all $n$. Also, GRH is equivalent to the
truth of $\RHgch$ for all characters $\gch$.

\bigskip\noindent{\sc Hypothesis $\RHsimgch$}:
{\it If $\rho=\beta+i\gamma$ is a {\tt simple} zero of $L(s,\gch)$ such that $\beta>0$, 
then $\beta=\frac12$.}

\bigskip\noindent{\sc Hypothesis $\RHsimstar$}:
{\it Hypothesis $\RHsimgch$ is true for {\tt every} character $\chi$.}

\bigskip\noindent The latter
two hypotheses lie at the heart of our work. $\RHsimgch$ is
a weak form of $\RHgch$ which asserts that the \emph{simple}
zeros of $L(s,\chi)$ in the critical strip all lie on the critical line;
nothing is assumed about the
trivial zeros of $L(s,\gch)$ (which are all simple)
or any zeros of multiplicity two or more
in the critical strip (that is, \emph{non-simple zeros}).

\bigskip\noindent{\sc Hypothesis $\LHgch$}:
{\it The function $L(s,\gch)$ satisfies the {\bf Lindel\"of bound}
\be\label{eq:LHbd}
L(\tfrac12+it,\gch)\llsym{ \gchm }\tau^\eps\qquad(t\in\R),
\ee
where $q\ge 1$ is the modulus of $\chi$, and $\tau=\tau(t)\defeq |t|+10$.}

\bigskip\noindent In particular, note that $\LHone$ is the classical \emph{Lindel\"of hypothesis}
for $\zeta(s)$. More generally, $\LHgch$ is (a weak form of) the \emph{generalized Lindel\"of hypothesis} for $L(s,\gch)$.

\bigskip\noindent{\sc Hypothesis $\LHstar$}:
{\it The hypothesis $\LHgch$ holds for {\tt every} character $\chi$.}

\begin{remark*}
The most important hypothesis from the perspective
of the present paper, namely $\RHdagsimpch$, can only be formulated
after some notation has been introduced; see \S\ref{sec:notation} below.
\end{remark*}

\section{Notation and statement of results}\label{sec:notation}

Following Riemann, the letter $s$ always denotes a complex variable, and we write
$\sigma\defeq\Re(s)$, and $t\defeq\Im(s)$. As in \eqref{eq:LHbd}, 
we put $\tau\defeq|t|+10$ for all $t\in\R$.

Any implied constants in the
symbols $\ll$, $O$, etc., may depend (where obvious) on the small parameter $\eps>0$;
any dependence on other parameters is indicated explicitly by the notation.
For example, \eqref{eq:LHbd} asserts that, for any $\eps>0$, the bound
$|L(\tfrac12+it,\gch)|\le C\tau^\eps$ holds for all $t\in\R$
with some constant $C>0$ that depends only on $ \gchm $ and $\eps$.

For an arbitrary character $\gch$, we make extensive use of the
function $D(s,\gch)$ defined in the half-plane $\{\sigma>1\}$ by
\[
D(s,\gch)\defeq\frac{L'(s,\gch)^2}{L(s,\gch)}=\sum_{n\in\N}\frac{\ell(n)\gch(n)}{n^s}
\]
and extended analytically to the complex plane.
Here, $\ell$ is the arithmetical function given by
\[
\ell(n)\defeq(\Lambda*\log)(n)=\ssum{a,b\in\N\\ab=n}\Lambda(a)\log b.
\]
Note that $0\le\ell(n)\le(\log n)^2$.
The function $D(s,\gch)$ is meromorphic with
a simple pole of residue $L'(\rho,\gch)$ at every simple zero $\rho$
of $L(s,\gch)$, and a pole of order three at $s=1$ 
in the case that $\gch$ is principal.
On the other hand, $D(s,\gch)$ is analytic in a neighborhood of any non-simple
zero $\rho$ of $L(s,\gch)$.

For a \emph{primitive} character $\pch$ mod~$\pchm$, we denote
\be\label{eq:red}
\kappa\defeq\begin{cases}
0&\quad\hbox{if $\pch(-1)=+1$,}\\
1&\quad\hbox{if $\pch(-1)=-1$,}\\
\end{cases}
\quad\tau(\pch)\defeq\sum_{a\bmod \pchm}\pch(a)\e(a/\pchm),
\quad\epsilon\defeq\frac{\tau(\pch)}{i^\kappa\sqrt{\pchm}},
\ee
where $\e(u)\defeq\er^{2\pi iu}$ for all $u\in\R$. The function defined by
\[
\cM_{\pch}(s)\defeq\epsilon\,2^s\pi^{s-1}\pchm^{1/2-s}
\Gamma(1-s)\sin\tfrac{\pi}{2}(s+\kappa)
\]
is familiar and plays an important r\^ole in analytic number theory, appearing as it does in the
asymmetric form of the functional equation:
\[
L(s,\pch)=\cM_{\pch}(s)L(1-s,\opch).
\]

\begin{lemma}\label{lem:Xxexpansion}
Let $\cI$ be a compact interval in $\R$. Uniformly for $c\in\cI$ and $t\ge 1$,
we have
\[
\cM_{\pch}(1-c-it)=\tau(\pch) \pchm^{c-1}\er^{-\pi i/4}
\exp\Big(it\log\Big(\frac{\pchm t}{2\pi\er}\Big)\Big)
\Big(\frac{t}{2\pi}\Big)^{c-1/2}\big\{1+O_\cI(t^{-1})\big\}.
\]
\end{lemma}

\begin{proof}
See Banks \cite[Lemma 2.1]{Banks3}.
\end{proof}

The next result is a variant of Conrey, Ghosh, and Gonek \cite[Lemma~1]{ConGhoGon}.
The proof relies on Gonek~\cite[Lemma~2]{Gonek} and
is based on the stationary phase method.

\begin{lemma}\label{lem:ConGhoGon}
Uniformly for $v>0$ and $c\in[\frac{1}{10},2]$, we have
\[
\frac{1}{2\pi i}\int_{c+i}^{c+iT}v^{-s}\cM_{\pch}(1-s)\,ds
=\begin{cases}
\tau(\pch)\pchm^{-1}\e(-v/\pchm)+E(\pchm,T,v)
&\quad\hbox{if $\frac{\pchm}{2\pi}<v\le \frac{\pchm T}{2\pi}$},\\ \\
E(\pchm,T,v)&\quad\hbox{otherwise},\\
\end{cases}
\]
where
\[
E(\pchm,T,v)\ll\frac{\pchm^{c-1/2}}{v^c}
\bigg(T^{c-1/2}+\frac{T^{c+1/2}}{|T-2\pi v/\pchm|+T^{1/2}}+1\bigg).
\]
\end{lemma}

\begin{proof}
See Banks \cite[Lemma 2.3]{Banks3}.
\end{proof}

Recall the Laurent series expansion of $\zeta(s)$ at $s=1$
(see, e.g., \cite[Prop.~10.3.19]{Cohen}):
\[
\zeta(s)=\frac1{s-1}+\sum_{n=0}^\infty\frac{(-1)^n}{n!}\,\gamma_n(s-1)^n,
\]
where $\{\gamma_n\}$ are the \emph{Stieltjes constants} given by
\[
\gamma_n\defeq\lim_{x\to\infty}
\bigg(\sum_{k\le x}\frac{(\log k)^n}{k}-\frac{(\log x)^{n+1}}{n+1}\bigg)
\qquad(n\ge 0).
\]
In particular, $\gamma_0$ is the \emph{Euler-Mascheroni constant}. For any $ \gchm $, let
\[
g_q(s)\defeq\prod_{p\,\mid\, \gchm }(1-p^{-s}),
\]
and denote
\dalign{
P_{\gchm} (X)&\defeq\tfrac12g_{\gchm} (1)X(\log X)^2-(1+\gamma_0)g_{\gchm} (1)X\log X\\
&\qquad+~\big((1+\gamma_0+\gamma_0^2+3\gamma_1)g_{\gchm} (1)
-\gamma_0g_{\gchm} '(1)-\tfrac12g_{\gchm} ''(1)\big)X.
}
This function is defined so that
\be\label{eq:residue}
\mathop{{\rm Res}}\limits_{s=1}D(s,\gch_{0, \gchm })\frac{X^s}{s}=P_{\gchm} (X)
\qquad(X>0),
\ee
where $\gch_{0, \gchm }$ is the principal character mod~$ \gchm $.

The following hypothesis on a primitive character $\pch$, despite its quite technical formulation, 
is crucial for understanding the relationship between simple zeros of different
Dirichlet $L$-functions.

\bigskip\noindent{\sc Hypothesis $\RHdagsimpch$}:
{\it The primitive character $\pch$ mod~$\pchm$
satisfies $\RHsimpch$, and for any rational
number $\xi=h/k$ with $h,k>0$ and $(h,k)=1$,
any $\euB\in C_c^\infty(\R^+)$, and  $X\ge 10$, we have
\be\label{eq:eureka}
\ssum{\rho=\frac12+i\gamma\\\gamma>0}\xi^{-\rho}
L'(\rho,\pch\cdot\gch_{0,\pchm k})
\cM_{\opch}(1-\rho)\euB\(\frac{\gamma}{2\pi\xi X}\)
-C_{\pch,\xi}\cdot F_{\pchm,k}(X)
\llsym{\pchm,\xi,\euB}X^{1/2+\eps},
\ee
where the sum runs over simple zeros 
of $L(s,\pch)$ in the critical strip,
\be\label{eq:CXxi-defn}
C_{\pch,\xi}\defeq\begin{cases}
\displaystyle{\frac{\opch(h)\pch(k)\mu(k)}{\phi(\pchm k)}}
&\quad\hbox{if $(h,\pchm k)=1$},\\
0&\quad\hbox{otherwise},\\
\end{cases} 
\ee
and
\be\label{eq:FXxiX-defn}
F_{\pchm,k}(X)\defeq\int_0^\infty\euB(u/(\pchm X))P'_{\pchm k}(u)\,du.
\ee
}

Observe that $\RHdagsimpch$ is formulated
entirely in terms of the simple zeros of a single $L$-function $L(s,\pch)$;
the hypothesis asserts nothing about any zeros of higher order
that might exist.
For characters $\pch_1\ne\pch_2$,
there is no reason \emph{a priori} to expect that the hypotheses
${\tt RH}_{\tt sim}^\dagger[\pch_1]$
and ${\tt RH}_{\tt sim}^\dagger[\pch_2]$ are related. However,
our principal theorem reveals a close connection.

%\begin{remarks*}
%\cco{Hypothesis $\RHdagsimpch$ is true for all primitive characters $\pch$ 
%under the generalized Riemann hypothesis.
%It is important to observe that $\RHdaggergch$ has been formulated
%entirely in terms of the simple zeros of a single $L$-function
%$L(s,\pch)$; its validity depends only on the horizontal and vertical
%distribution of the simple zeros.
%If $\pch$ and $\pcho$ are \emph{different} primitive characters,
%then there is no reason, a priori, that the hypotheses
%$\RHdagsimpch$ and $\RHdagsimpcho$ should be related to one another. Nevertheless,
%we show that these hypotheses are equivalent if one assumes $\LHstar$.}
%\end{remarks*}
%
%Our principal theorem is the following.

\begin{theorem}\label{thm:main} 
Assume $\LHstar$. For each primitive character $\pch$, the
hypotheses $\RHdagsimpch$ and $\RHsimstar$ are equivalent.
\end{theorem}

Thus, assuming $\LHstar$, Theorem~\ref{thm:main} implies that
${\tt RH}_{\tt sim}^\dagger[\pch_1]
\Longleftrightarrow {\tt RH}_{\tt sim}^\dagger[\pch_2]$
for all primitive characters $\pch_1$~and~$\pch_2$.

\begin{corollary}\label{cor:A} 
Under {\rm GRH}, $\RHdagsimpch$ holds for every primitive character~$\pch$.
\end{corollary}

This is clear since GRH implies both $\LHstar$ and $\RHsimstar$.

\begin{corollary}\label{cor:main2} 
Assume $\LHstar$, and suppose $\RHdagsimpch$ is true for some
primitive character~$\pch$. Then the simple zeros of any Dirichlet
$L$-function $L(s,\gch)$ lie on the critical line. In particular,
Siegel zeros do not exist.
\end{corollary}

\section{Bounding $D(s,\gch)$}

For any real $\sigma_0>0$, let $\oR(\sigma_0)$ and $\cR(\sigma_0)$
be the closed regions defined by
\[
\oR(\sigma_0)\defeq\{s:\sigma\ge\sigma_0,~|s-1|\ge\tfrac{1}{25}\}
\mand
\cR(\sigma_0)\defeq\{s:\sigma\ge\sigma_0\};
\]
see Figure~\ref{fig:regionRsig} below. Let $\oV(\sigma_0)$ [resp.~$\cV(\sigma_0)$]
be the vector space of consisting of all
meromorphic functions~$F$ that satisfy, for every $\eps>0$, the bound
\be\label{eq:LHsetbd}
F(s)\llsym{F,\sigma_0}\tau^{\lambda(\sigma)+\eps},
\qquad\lambda(\sigma)\defeq\max\{0,\tfrac12-\sigma\},
\ee
uniformly at all points $s$ in $\oR(\sigma_0)$ [resp.~$\cR(\sigma_0)$].
\begin{figure}[H]
\includegraphics[width=4in]{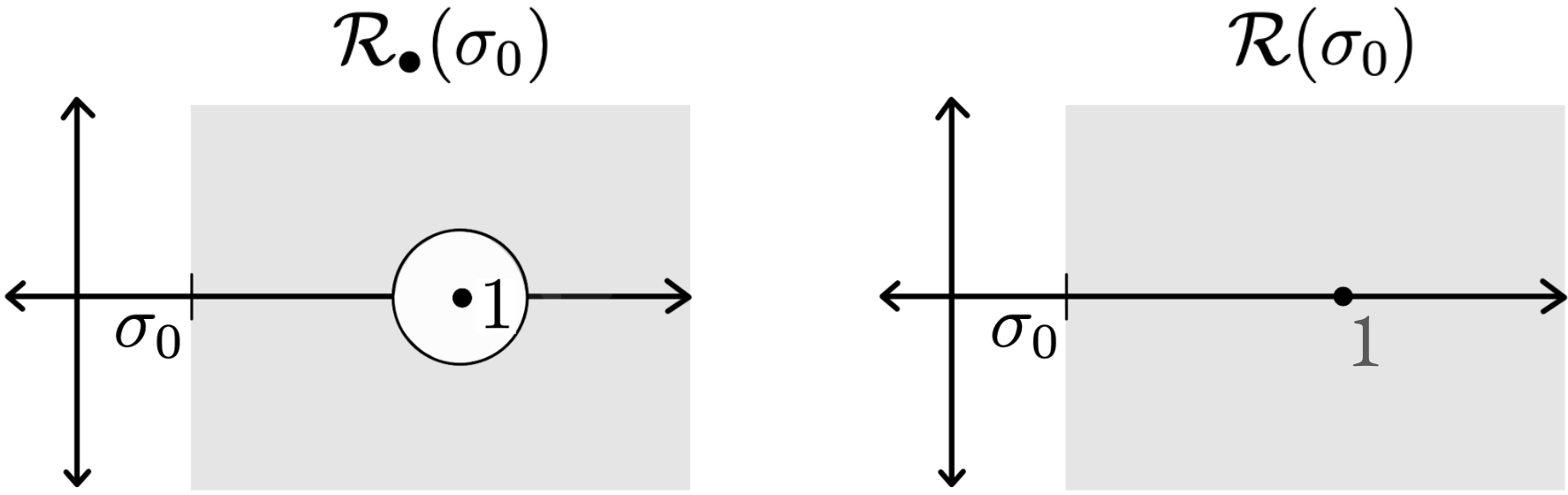}
\caption{}\label{fig:regionRsig}
\end{figure}

In this section, we consider the problem of bounding the function $D(s,\gch)$
introduced in~\S\ref{sec:notation}, where $\gch$ is an arbitrary
character mod~$ \gchm $. We denote by $\gch_{0, \gchm }$ the principal character mod~$ \gchm $,
which is also the indicator function of integers coprime to~$ \gchm $.
We also denote
\be\label{eq:defn-fchis}
\tL(s,\gch)\defeq L(s,\gch)-\frac{\delta_{\gch} c_{\gchm} }{s-1}
\ee
with
\[
\delta_{\gch}\defeq\begin{cases}
1&\quad\hbox{if $\gch$ is principal},\\
0&\quad\hbox{if $\gch$ is nonprincipal},\\
\end{cases}
\mand
c_{\gchm} \defeq\frac{\phi( \gchm )}{ \gchm },
\]
where $\phi$ is the Euler totient function.
For any choice of $\gch$, the function $\tL(s,\gch)$ extends to an entire
function. Since \eqref{eq:defn-fchis} implies the estimate
\be\label{eq:cpt-est}
\tL(s,\gch)=L(s,\gch)+O(1)\qquad(s\in\oR(\tfrac1{25})),
\ee
the hypothesis $\LHgch$ admits the following equivalent formulation.

\bigskip\noindent{\sc Hypothesis $\LHgch$}:
{\it The function $\tL(s,\gch)$ satisfies the Lindel\"of bound}
\be\label{eq:LHbd2}
\tL(\tfrac12+it,\gch)\llsym{ \gchm }\tau^\eps\qquad(t\in\R).
\ee

\begin{lemma}\label{lem:cicada-1}
The following statements are equivalent:
\begin{itemize}
\item[$(i)$] $\tL(s,\gch)$ belongs to $\cV(\frac1{25})$;
\item[$(ii)$] $\LHgch$ is true.
\end{itemize}
\end{lemma}

\begin{remark*}
The number $\frac1{25}$ can be replaced by any positive absolute constant.
\end{remark*}

\begin{proof}
If $\tL(s,\gch)$ belongs to $\cV(\frac1{25})$, then applying
\eqref{eq:LHsetbd} with $\sigma\defeq\frac12$ and $\sigma_0\defeq \frac1{25}$,
we obtain \eqref{eq:LHbd2} at once. This shows that $(i)\Longrightarrow(ii)$.

Conversely, when $\LHgch$ is true, \eqref{eq:LHbd2} holds.
Also, $\tL(2+it,\gch)\ll 1$ holds unconditionally. Since $\tL(s,\gch)$ is entire,
the Phragmen-Lindel\"of theorem gives
\be\label{eq:PhrgLd}
\tL(s,\gch)\llsym{ \gchm }\tau^\eps\qquad(\sigma\ge\tfrac12).
\ee

Next, suppose that $\sigma\in[\frac1{25},\frac12]$.
Replacing $s$ by $1-s$ in \eqref{eq:PhrgLd}, we have
\be\label{eq:L(1-s)chis}
L(1-s,\gch)\llsym{ \gchm }\tau^\eps
\mand
L(1-s,\overline\gch)\llsym{ \gchm }\tau^\eps.
\ee
Moreover, we have
\be\label{eq:fnceqbd}
L(s,\gch)\llsym{ \gchm }\tau^{1/2-\sigma}L(1-s,\overline\gch).
\ee
Indeed, for a primitive character $\pch$ mod~$\pchm$, the bound
\[
L(s,\pch)\llsym{\pchm}\tau^{1/2-\sigma}|L(1-s,\opch)|
\]
follows from \cite[Cors.~10.5 and~10.10]{MontVau}. More generally,
if $\gch$ is induced from $\pch$, then
(with $s$ as above) we have
\[
L(s,\gch)\,\mathop{\asymp}\limits_{\gchm} \,L(s,\pch)
\mand
L(1-s,\ogch)\,\mathop{\asymp}\limits_{\gchm} \,L(1-s,\opch),
\]
and \eqref{eq:fnceqbd} follows.
Using \eqref{eq:cpt-est}, \eqref{eq:L(1-s)chis}, and
\eqref{eq:fnceqbd}, we get that
\[
\tL(s,\gch)=L(s,\gch)+O(1)
\llsym{ \gchm } \tau^{1/2-\sigma}\big|L(1-s,\overline\gch)\big|+O(1)
\llsym{ \gchm }\tau^{1/2-\sigma+\eps}\quad(\sigma\in[\tfrac1{25},\tfrac12]).
\]
Combining this with \eqref{eq:PhrgLd}, we see that
$\tL(s,\gch)\in\cV(\frac1{25})$, and so $(ii)\Longrightarrow(i)$.
\end{proof}

\begin{lemma}\label{lem:cicada-2}
If $\sigma_0>0$ and $F\in\cV(\sigma_0)$, then $F'\in\cV(\sigma_0+\frac1{25})$.
\end{lemma}

\begin{proof}
Let $s\in\cR(\sigma_0+\frac1{25})$.
For any $\eps\in(0,\frac2{25})$, let $\cC$ be the circle in the complex plane
with center $s$ and radius~$\frac\eps2$. Since $F\in\cV(\sigma_0)$
and each $z\in\cC$ satisfies
$\Re(z)\ge \sigma-\tfrac\eps2\ge\sigma_0$ (and thus $z\in\cR(\sigma_0)$), we have
\[
F(z)\llsym{F,\sigma_0}\tau^{\max\{0,1/2-\Re(z)\}+\eps/2}
\le\tau^{\lambda(\sigma)+\eps}.
\]
By the Cauchy integral formula,
\[
\big|F'(s)\big|=\bigg|\frac{1}{2\pi i}\oint_\cC\frac{F(z)\,dz}{(z-s)^2}\bigg|
\le\frac{2}{\eps}\max_{z\in\cC}\big|F(z)\big|\llsym{F,\sigma_0}
\tau^{\lambda(\sigma)+\eps},
\]
and the lemma follows.
\end{proof}

Combining Lemmas~\ref{lem:cicada-1} and~\ref{lem:cicada-2}, the next result is immediate.

\begin{corollary}\label{cor:Lbds-combined} Under $\LHgch$, we have
\begin{alignat}{3}
\nonumber
L(s,\gch)&=\delta_{\gch} c_{\gchm} (s-1)^{-1}+O_{ \gchm }(\tau^{\lambda(\sigma)+\eps})
\qquad&(s\in\cR(\tfrac1{25})),\\
\label{eq:sink2}
L'(s,\gch)&=-\delta_{\gch} c_{\gchm} (s-1)^{-2}+O_{ \gchm }(\tau^{\lambda(\sigma)+\eps})
\qquad&(s\in\cR(\tfrac2{25})),\\
\nonumber
L''(s,\gch)&=2\delta_{\gch} c_{\gchm} (s-1)^{-3}+O_{ \gchm }(\tau^{\lambda(\sigma)+\eps})
\qquad&(s\in\cR(\tfrac3{25})).
\end{alignat}
In particular, if $\rho=\beta+i\gamma$ is a nontrivial zero of $L(s,\gch)$, then
\begin{alignat}{3}
\label{eq:sink4}
L'(\rho,\gch)&\llsym{ \gchm }\tau^{\lambda(\beta)+\eps}
\qquad&(\rho\in\cR(\tfrac2{25})),\\
\label{eq:sink5}
L''(\rho,\gch)&\llsym{ \gchm }\tau^{\lambda(\beta)+\eps}
\qquad&(\rho\in\cR(\tfrac3{25})).
\end{alignat}
\end{corollary}

\begin{lemma}\label{lem:cicada-4}
Assume $\RHsimgch$ and $\LHgch$.
For any nontrivial zero $\rho$ of $L(s,\gch)$, the function $f_\rho$ defined by
\[
f_\rho(s)\defeq\begin{cases}
\displaystyle\frac{L'(s,\gch)-L'(\rho,\gch)}{s-\rho}
&\quad\hbox{if $s\ne\rho$},\\
L''(\rho,\gch)&\quad\hbox{if $s=\rho$},
\end{cases}
\]
satisfies the bound
\be\label{eq:rose}
f_\rho(s)\llsym{ \gchm }\tau^{\lambda(\sigma)+\eps}+\delta_{\gch} c_{\gchm} |s-1|^{-2}\tau^\eps
\ee
uniformly for $s\in\cR(\tfrac4{25})$.
\end{lemma}

\begin{proof}
We start with the fact that
\be\label{eq:wessex}
L'(\rho,\gch)\llsym{ \gchm }\tau^{\eps/2}
\ee
for any nontrivial zero $\rho$.
Indeed, if $\rho$ is non-simple, then we have $L'(\rho,\gch)=0$. On the other
hand, if $\rho=\beta+i\gamma$
is a simple (and nontrivial) zero, then $\beta=\frac12$ under $\RHsimgch$, whence
\eqref{eq:sink4} immediately implies \eqref{eq:wessex}.

Now, let $s\in\cR(\frac4{25})$, $s\ne 1$. We consider three different cases.

First, suppose $s=\rho$. Then $f_\rho(s)=L''(\rho,\gch)$, and
$\rho=s\in\cR(\frac4{25})$, hence \eqref{eq:rose}
follows directly from \eqref{eq:sink5}.

Next, suppose $0<|s-\rho|\le\frac1{50\log\tau}$, and write
\[
f_\rho(s)=\frac{1}{2\pi i}\oint_\cC\frac{f_\rho(z)}{(z-s)}\,dz
=\frac{1}{2\pi i}\oint_\cC\frac{L'(z,\gch)-L'(\rho,\gch)}{(z-\rho)(z-s)}\,dz
\]
where $\cC$ is the circle in the complex plane with center $s$
and radius $\frac1{25\log\tau}$, oriented counterclockwise. 
It is straightforward to check that
\[
\min\{|z-s|,|z-\rho|\}\ge\tfrac1{50\log\tau}\qquad(z\in\cC),
\]
and therefore
\[
f_\rho(s)\ll(\log\tau)
\Big\{\max\limits_{z\in\cC}\big|L'(z,\gch)\big|+\big|L'(\rho,\gch)\big|\Big\}
\llsym{ \gchm }(\log\tau)\max\limits_{z\in\cC}\big|L'(z,\gch)\big|
+\tau^\eps,
\]
where we used \eqref{eq:wessex} in the second step.
To prove \eqref{eq:rose} in this case, it is enough to show that
\be\label{eq:wessex2}
L'(z,\gch)\llsym{ \gchm }\tau^{\lambda(\sigma)+\eps/2}\qquad(z\in\cC).
\ee
Let $z=x+iy$ be a number in $\cC$.
Since $|x-\sigma|\le|z-s|=\frac1{25\log\tau}$, we have
$\tau^{\lambda(x)}\asymp\tau^{\lambda(\sigma)}$ and also $z\in\cR(\frac15)$;
hence \eqref{eq:wessex2} follows from \eqref{eq:sink2}.

Finally, suppose $|s-\rho|\ge\frac1{50\log\tau}$. In this case,
\[
f_\rho(s)\ll(\log\tau)\Big\{\big|L'(s,\gch)\big|+\big|L'(\rho,\gch)\big|\Big\}
\llsym{ \gchm }(\log\tau)\big|L'(s,\gch)\big|+\tau^\eps.
\]
Since $s\in\cR(\frac4{25})$, the required bound \eqref{eq:rose} follows from \eqref{eq:sink2}.
\end{proof}

\begin{lemma}\label{lem:cicada-5} Under $\RHsimgch$ and $\LHgch$, we have
\be\label{eq:DXsbd}
D(s,\gch)+\delta_{\gch}\frac{L'(s,\gch)}{s-1}
\llsym{ \gchm }\tau^{\lambda(\sigma)+\eps}+\delta_{\gch} |s-1|^{-2}\tau^\eps
+|\sigma-\tfrac12|^{-1}\tau^\eps
\ee
uniformly for $s\in\cR(\frac15)$.
\end{lemma}

\begin{proof} Suppose $\gch$ is induced from the primitive character $\pch$.
Unconditionally, we have (see, e.g., \cite[Lems.~12.1 and~12.6]{MontVau}):
\[
\frac{L'}{L}(s,\pch)=-\frac{\delta_{\pch}}{s-1}
+\ssum{\rho\\|\gamma-t|\le 1}\frac{1}{s-\rho}+O_{\pchm}(\log\tau)
\qquad(\sigma\in[-1,2]),
\]
where the sum runs over nontrivial zeros $\rho$ of $L(s,\pch)$.
Taking into account that
\[
\frac{L'}{L}(s,\gch)=\frac{L'}{L}(s,\pch)
+\sum_{p\,\mid\, \gchm }\frac{\pch(p)\log p}{p^s-\pch(p)},
\]
it follows that 
\[
\frac{L'}{L}(s,\gch)=-\frac{\delta_{\gch}}{s-1}
+\ssum{\rho\\|\gamma-t|\le 1}\frac{1}{s-\rho}+O_{\gchm} (\log\tau)
\qquad(\sigma\in[\tfrac15,2]).
\]
Multiplying by $L'(s,\gch)$, we get that
\be\label{eq:runningkids}
\tD(s,\gch)=\ssum{\rho\\|\gamma-t|\le 1}\frac{L'(s,\gch)}{s-\rho}
+O_{\gchm} (|L'(s,\gch)|\log\tau)\qquad(\sigma\in[\tfrac15,2]),
\ee
where $\tD(s,\gch)$ is the function defined on the left side of \eqref{eq:DXsbd}.
The error term in \eqref{eq:runningkids} is acceptable in view of \eqref{eq:sink2}.
To bound the sum in \eqref{eq:runningkids}, observe that
\[
\ssum{\rho\\|\gamma-t|\le 1}\frac{L'(s,\gch)}{s-\rho}
=\ssum{\rho~\text{\tt non-simple}\\|\gamma-t|\le 1}f_\rho(s)
+\ssum{\rho~\text{\tt simple}\\|\gamma-t|\le 1}\frac{L'(\rho,\gch)}{s-\rho}.
\]
As these sums all involve $\ll\log  \gchm \tau$ zeros, we obtain \eqref{eq:DXsbd} by
applying Lemma~\ref{lem:cicada-4} together with \eqref{eq:wessex}, taking into
account that $|s-\rho|\ge|\sigma-\frac12|$ for any simple zero $\rho$
(under $\RHsimgch$). This completes the proof.
\end{proof}

The next result, used in the proof of Theorem~\ref{thm:vonMangoldt-twist}
below, is conditional on $\LHgch$ but not on $\RHsimgch$.

\begin{lemma}\label{lem:trade}
Assume $\LHgch$.
For any $t\ge 2$, there is a real number $t_*\in[t,t+1]$ such that
\[
D(\sigma\pm it_*,\gch)\llsym{ \gchm }
\tau^{\lambda(\sigma)+\eps}\qquad(\sigma\in[\tfrac2{25},2]).
\]
\end{lemma}

\begin{proof}
By \cite[Lemmas 12.2 and 12.7]{MontVau}, there is a number $t_*\in[t,t+1]$ such that
\[
\frac{L'}{L}(\sigma\pm it_*,\gch)\ll(\log  \gchm  t)^2\qquad(\sigma\in[-1,2]).
\]
Multiplying by $L'(\sigma\pm it_*,\gch)$ and using \eqref{eq:sink2}, the result follows.
\end{proof}

\section{Criteria for $\RHsimgch$}

\medskip\begin{theorem}\label{thm:ultraclean}
Let $\gch$ be a  character mod~$ \gchm $ that
satisfies $\LHgch$. Then the following are equivalent:
\begin{itemize}
\item[$(i)$]  $\RHsimgch$ is true;
\item[$(ii)$] For any $X\ge 10$, we have
\be\label{eq:bird1}
\sum_{n\le X}\ell(n)\gch(n)=\delta_{\gch} P_{\gchm} (X)+O_{ \gchm }(X^{1/2+\eps});
\ee
\item[$(iii)$] For any $\euB\in C_c^\infty(\R^+)$ and $X\ge 10$, we have
\be\label{eq:bird2}
\sum_n\ell(n)\gch(n)\euB(n/X)
=\delta_{\gch}\int_0^\infty\euB(u/X)P'_{\gchm} (u)\,du+O_{ \gchm ,\euB}(X^{1/2+\eps}).
\ee
\end{itemize}
\end{theorem}

\begin{proof} $(i)\Rightarrow(ii)$.
Let $T\defeq\sqrt{X}$, $\sigma_0\defeq 1+\frac1{\log X}$,
and $\sigma_1\defeq\frac12+\frac1{\log X}$.
Let $\cC$ be the rectangular contour in $\C$ consisting of the four
directed line segments:
\dalign{
\cC_1:\quad&\sigma_0-iT\longrightarrow\sigma_0+iT,\\
\cC_2:\quad&\sigma_0+iT\longrightarrow\sigma_1+iT,\\
\cC_3:\quad&\sigma_1+iT\longrightarrow\sigma_1-iT,\\
\cC_4:\quad&\sigma_1-iT\longrightarrow\sigma_0-iT.
}
Using Perron's formula (see, e.g., \cite[Thm.~5.2 and Cor.~5.3]{MontVau}),
it follows that
\[
\sum_{n\le X}\ell(n)\gch(n)=\frac{1}{2\pi i}\int_{\cC_1}
D(s,\gch)\,\frac{X^s}{s}\,ds+O(X^{1/2}(\log X)^3),
\]
where the implied constant is absolute.
Under $\RHsimgch$, the function $D(s,\gch)$ is analytic in
the half-plane $\{\sigma>\frac12\}$ unless $\gch=\mathbbm 1$, in which case
there is a triple pole at $s=1$. Using Cauchy's theorem and \eqref{eq:residue}, we see that
\[
\frac{1}{2\pi i}\bigg(\int_{\cC_1}+\int_{\cC_2}
+\int_{\cC_3}+\int_{\cC_4}\bigg)D(s,\gch)\,\frac{X^s}{s}\,ds
=\frac{1}{2\pi i}\oint_\cC D(s,\gch)\,\frac{X^s}{s}\,ds=\delta_{\gch} P_{\gchm} (X);
\]
consequently,
\[
\sum_{n\le X}\ell(n)\gch(n)=\delta_{\gch} P_{\gchm} (X)-
\frac{1}{2\pi i}\bigg(\int_{\cC_2}+\int_{\cC_3}
+\int_{\cC_4}\bigg)D(s,\gch)\,\frac{X^s}{s}\,ds+O_\eps(X^{1/2+\eps}).
\]
By Corollary~\ref{cor:Lbds-combined} (bound \eqref{eq:sink2}) and
Lemma~\ref{lem:cicada-5} (bound \eqref{eq:DXsbd}), the bound
\[
D(s,\gch)\llsym{ \gchm }\tau^{\eps/2}\log X
\]
holds uniformly for any $s$ on the segments $\cC_2$, $\cC_3$, and $\cC_4$; consequently,
\[
\int_{\cC_j}D(s,\gch)\,\frac{X^s}{s}\,ds
\llsym{ \gchm }\frac{\log X}{T^{1-\eps/2}}\int_{\sigma_1}^{\sigma_0}X^{\sigma}\,d\sigma
\ll X^{1/2+\eps}\qquad(j=2\text{~or~}4),
\]
and 
\[
\int_{\cC_3}D(s,\gch)\,\frac{X^s}{s}\,ds
\llsym{ \gchm }X^{1/2}\log X\int_{-T}^T\frac{\tau^{\eps/2}}{1+|t|}\,dt
\ll X^{1/2+\eps}.
\]
Putting everything together, we obtain \eqref{eq:bird1}.

\bigskip\noindent $(ii)\Rightarrow(i)$. Let
\be\label{eq:var-defns}
S(u)\defeq\sum_{n\le u}\ell(n)\gch(n)=\delta_{\gch} P_{\gchm} (u)+E(u),
\qquad\text{where}\quad
E(u)\llsym{ \gchm }u^{1/2+\eps}.
\ee
In the region $\{\sigma>1\}$ we have
\be\label{eq:spraybottle}
D(s,\gch)=\sum_{n=1}^\infty\frac{\ell(n)\gch(n)}{n^s}
=\delta_{\gch}\int_1^\infty u^{-s}\,P'_{\gchm} (u)\,du
+\int_{1^-}^\infty u^{-s}\,dE(u).
\ee
The first integral in \eqref{eq:spraybottle} evaluates to
\be\label{eq:Eschi-defn}
D_\infty(s,\gch)\defeq\frac{g_{\gchm} (1)}{(s-1)^3}-\frac{\gamma_0g_{\gchm} (1)}{(s-1)^2}+
\frac{(\gamma_0^2+3\gamma_1)g_{\gchm} (1)-\gamma_0 g_{\gchm} '(1)-\tfrac12g_{\gchm} ''(1)}{s-1},
\ee
which is the singular part of $D(s,\gch)$ at $s=1$ in the case that $\gch=\gch_{0, \gchm }$.
Integrating by parts, we see that the second integral in \eqref{eq:spraybottle}
is equal to
\[
\int_{1^-}^\infty u^{-s}\,dE(u)
=-P_{\gchm} (1)+s\int_1^\infty u^{-s-1}E(u)\,du.
\]
Since $E(u)\ll_{ \gchm }u^{1/2+\eps}$, the right side continues analytically
to the half-plane $\{\sigma>\frac12+\eps\}$; consequently, the function
\[
D(s,\gch)-\delta_{\gch} D_\infty(s,\gch)
\]
continues analytically to the same half-plane. In particular,
$L(s,\gch)$ has no simple zeros in that region.
Taking $\eps\to 0^+$, this verifies $\RHsimgch$.

\bigskip\noindent $(ii)\Rightarrow(iii)$. Using \eqref{eq:bird1}
and \eqref{eq:var-defns}, we have by partial summation:
\[
\sum_n\ell(n)\gch(n)\euB(n/X)
=\delta_{\gch}\int_0^\infty\euB(u/X)P'_{\gchm} (u)\,du+\int_0^\infty\euB(u/X)\,dE(u).
\]
Using integration by parts and the bound $E(u)\ll_{ \gchm }u^{1/2+\eps}$,
it is easily shown that the second integral is $\ll_{ \gchm ,\euB}X^{1/2+\eps}$,
hence we have \eqref{eq:bird2}.

\bigskip\noindent $(iii)\Rightarrow(i)$. For each $\euB\in C_c^\infty(\R^+)$,
we define
\[
\widehat\euB(s)=\int_0^\infty\euB(u)u^{s-1}\,du\qquad(s\in\C),
\]
and we fix a number $d_\euB<1$ such that $\euB(u)=0$ if $u\le d_\euB$ or $u\ge d_\euB^{-1}$.
Let
\[
f_\euB(x)\defeq\sum_n\ell(n)\gch(n)\euB(n/x),\qquad
g_\euB(x)\defeq\delta_{\gch}\int_{xd_\euB}^{x/d_\euB}\euB(u/x)P'_{\gchm} (u)\,du,
\]
and observe that $f_\euB(x)=0$ for all $x\le d_\euB$. By \eqref{eq:bird2}, we have
\be\label{eq:trim}
f_\euB(x)-g_\euB(x)\llsym{ \gchm ,\euB}x^{1/2+\eps}\qquad(x\ge 10).
\ee

In the half-plane $\{\sigma>1\}$, it is straightforward to show that
\dalign{
\int_{d_\euB}^\infty f_\euB(x)x^{-s-1}\,dx&=\widehat\euB(s)\cdot D(s,\gch),\\
\int_{d_\euB}^\infty g_\euB(x)x^{-s-1}\,dx
&=\widehat\euB(s)\cdot\delta_{\gch} \int_{d_\euB}^\infty u^{-s}P'_{\gchm} (u)\,du.
}
The latter integral can be split as $D_0(s,\gch)+D_\infty(s,\gch)$, where
\[
D_0(s,\gch)\defeq\int_{d_\euB}^1 u^{-s}P'_{\gchm} (u)\,du
\mand
D_\infty(s,\gch)\defeq\int_1^\infty u^{-s}P'_{\gchm} (u)\,du.
\]
The first integral defines $D_0(s,\gch)$ as an entire function in the complex plane.
For $\sigma>1$, $D_\infty(s,\gch)$ is explicitly given by \eqref{eq:Eschi-defn}. We see
that $D_\infty(s,\gch)$ analytically continues to the complex plane except
for a possible pole at $s=1$. As mentioned above, $D_\infty(s,\gch)$ is the singular part
of $D(s,\gch)$ in $s=1$ when $\gch=\gch_{0, \gchm }$. As $\widehat\euB(s)$
is entire, it follows that any pole of the function
\[
I(s,\gch)\defeq\int_{d_\euB}^\infty \big\{f_\euB(x)-g_\euB(x)\big\}x^{-s-1}\,dx
=\widehat\euB(s)\cdot\big\{D(s,\gch)-D_0(s,\gch)-D_\infty(s,\gch)\big\}
\]
with real part $\sigma>0$ must occur at a \emph{simple} zero $\rho$ of $L(s,\gch)$.
Moreover, if $\rho$ is such a zero, then choosing $\euB$ so that $\widehat\euB(\rho)\ne 0$,
we see that $I(s,\gch)$ does indeed have a pole at $s=\rho$.

On the other hand, in view of \eqref{eq:trim}, the integral defining $I(s,\gch)$
converges absolutely in the half-plane $\{\sigma>\frac12+\eps\}$, so $I(s,\gch)$
cannot have a pole there for any choice of $\euB$. Taking $\eps\to 0^+$, we deduce
that $L(s,\gch)$ has no simple zeros in $\{\sigma>\frac12\}$, and $\RHsimgch$
is verified.
\end{proof}

\section{Twisted sums with $\ell(n)$}

\begin{theorem}\label{thm:vonMangoldt-twist}
Let $\gch$ be a character mod~$ \gchm $ induced from a primitive character $\pch$
mod~$\pchm$ that satisfies $\LHpch$. For any $\xi\in\R^+$ and $T>0$, 
\[
\ssum{\rho=\beta+i\gamma\\\beta\ge\frac12\\0<\gamma\le T}\xi^{-\rho}
L'(\rho,\gch)\cM_{\overline\pch}(1-\rho)
-\frac{\tau(\opch)}{\pchm}\sum_{n\le \pchm T/(2\pi\xi)}\ell(n)\gch(n)\e(-n\xi)
\llsym{ \gchm ,\xi}T^{1/2+\eps}+1.
\]
\end{theorem}

\begin{proof}
The result is trivial for $T<100$, so we assume $T\ge 100$ in what follows.
For any $u>0$, let
\[
\Sigma_1(u)\defeq\ssum{\rho=\beta+i\gamma\\\beta\ge\frac12\\0<\gamma\le u}
\xi^{-\rho}L'(\rho,\gch)\cM_{\opch}(1-\rho),\qquad
\Sigma_2(u)\defeq\frac{\tau(\opch)}{\pchm}
\hskip-5pt\ssum{n\le \pchm u/(2\pi\xi)}
\hskip-5pt\ell(n)\gch(n)\e(-n\xi).
\]
In this notation, the theorem (for $T\ge 100$) asserts that
\be\label{eq:Sigsum}
\Sigma_1(T)-\Sigma_2(T)
\llsym{ \gchm ,\xi}T^{1/2+\eps}.
\ee
In fact, to prove \eqref{eq:Sigsum} for any $T\ge 100$, it suffices to prove that
\be\label{eq:Sigsumstar}
\Sigma_1(T_*)-\Sigma_2(T_*)
\llsym{ \gchm ,\xi}T^{1/2+\eps}
\ee
holds for \emph{at least one} number $T_*\in[T,T+1]$. Indeed,
by Lemma~\ref{lem:Xxexpansion}, we have
$\cM_{\opch}(1-\rho)\ll (\pchm\gamma)^{1/2}$
uniformly for all nontrivial zeros $\rho=\beta+i\gamma$ of $L(s,\gch)$
such that $\gamma\ge 10$ (say). Consequently,
\[
\big|\Sigma_1(T_*)-\Sigma_1(T)\big|\le
\ssum{\rho=\beta+i\gamma\\\beta\ge\frac12\\T<\gamma\le T_*}
\big|\xi^{-\rho}\cM_{\opch}(1-\rho)\big|
\llsym{ \gchm ,\xi}T^{1/2}\log T
\]
since there at most $O(\log  \gchm  T)$ zeros with $T<\gamma\le T_*$.
Furthermore,
\[
\big|\Sigma_2(T_*)-\Sigma_2(T)\big|\le\frac{|\tau(\opch)|}{\pchm}
\ssum{\pchm T/(2\pi\xi)<n\le \pchm T_*/(2\pi\xi)}\big|\ell(n)\gch(n)\e(-n\xi/\pchm)\big|
\llsym{\pchm,\xi}(\log T)^2
\]
since $0\le\ell(n)\le(\log n)^2$.
These bounds make it clear that \eqref{eq:Sigsum}
and \eqref{eq:Sigsumstar} are equivalent, and the claim is proved.

By the preceding argument, and recalling Lemma~\ref{lem:trade},
for the proof of \eqref{eq:Sigsum} we can assume without loss of generality that
\be\label{eq:horiz-up}
D(\sigma\pm iT,\gch)\llsym{ \gchm }
T^{\lambda(\sigma)+\eps}\qquad(\sigma\in[\tfrac2{25},2]).
\ee
Moreover, by Lemma~\ref{lem:trade}, there is a number
$t_\circ\in[2,3]$ such that
\[
D(\sigma\pm it_\circ,\gch)\llsym{ \gchm }1\qquad(\sigma\in[\tfrac2{25},2]).
\]
For such $t_\circ$, it can be shown that
\be\label{eq:gosh1}
\Sigma_1(t_\circ)\llsym{ \gchm ,\xi}1
\mand
\Sigma_2(t_\circ)\llsym{ \gchm ,\xi}1;
\ee
see, e.g., the proof of \cite[Thm.~3.1]{Banks3}.
We fix $t_\circ$ and $T$ with these properties.
Put $c\defeq 1+\tfrac{1}{\log  \gchm  T}$, and let $b$ be any number
in the open interval $(\tfrac12-\tfrac{1}{\log  \gchm  T},\frac12)$
such that $L(s,\gch)\ne 0$ for $\sigma\in[b,\frac12)$ and $t\in[t_\circ,T]$.
Finally, let $\cC$ be the rectangular contour consisting of the four
directed line segments:
\dalign{
\cC_1:\quad&c+it_\circ\longrightarrow c+iT,\\
\cC_2:\quad&c+iT\longrightarrow b+iT,\\
\cC_3:\quad&b+iT\longrightarrow b+it_\circ,\\
\cC_4:\quad&b+it_\circ\longrightarrow c+it_\circ.
}
Our choices of $T$, $t_\circ$, $b$, and $c$ guarantee that
$D(s,\gch)$ has no singularity on the contour $\cC$.
By Cauchy's theorem, we have 
\dalign{
\Sigma_1(T)-\Sigma_1(t_\circ)
&=\frac{1}{2\pi i}\oint_{\cC}D(s,\gch)\,\xi^{-s}\cM_{\opch}(1-s)\,ds\\
&=\frac{1}{2\pi i}\bigg(\int_{\cC_1}+\int_{\cC_2}
+\int_{\cC_3}+\int_{\cC_4}\bigg)
D(s,\gch)\,\xi^{-s}\cM_{\opch}(1-s)\,ds\\
&=I_1+I_2+I_3+I_4\quad\text{(say)},
}
and thus by \eqref{eq:gosh1}, we have
\be\label{eq:Ibegin}
\Sigma_1(T)=I_1+I_2+I_3+I_4+O_{\gchm} (1).
\ee
We estimate the four integrals $I_j$ separately.

First, recalling the Dirichlet expansion
\[
D(s,\gch)\defeq\frac{L'(s,\gch)^2}{L(s,\gch)}=\sum_{n\in\N}\frac{\ell(n)\gch(n)}{n^s}
\qquad(\sigma>1),
\]
it is immediate that
\[
I_1=\sum_{n\in\N}\ell(n)\gch(n)\cdot\frac{1}{2\pi i}\int_{c+it_\circ}^{c+iT}
(n\xi)^{-s}\cM_{\opch}(1-s)\,ds.
\]
Applying Lemma~\ref{lem:ConGhoGon} with both $T$ and $t_\circ$, we derive the estimate
\dalign{
I_1&=\frac{\tau(\opch)}{\pchm}
\sum_{\pchm t_\circ/(2\pi\xi)<n\le \pchm T/(2\pi\xi)}\ell(n)\gch(n)\e(-n\xi/\pchm)
+O\bigg(\sum_n\ell(n)\big|E(\pchm,T,n\xi)\big|\bigg)\\
&=\Sigma_2(T)-\Sigma_2(t_\circ)
+O_{ \gchm ,\xi}\big(T^{c-1/2}(E_1+E_2)\big),
}
where
\[
E_1\defeq\sum_n\frac{\ell(n)}{n^c}
\mand
E_2\defeq\sum_n\frac{\ell(n)}{n^c}
\frac{T}{|T-2\pi n\xi/\pchm|+T^{1/2}}.
\]
Clearly,
\be\label{eq:I1 E1bd}
E_1=\frac{\zeta'(c)^2}{\zeta(c)}\ll \frac{1}{(c-1)^3}=(\log T)^3.
\ee
Also, setting $T_\circ\defeq T/(2\pi\xi)$, we have
\be\label{eq:I1 E2bd}
E_2\llsym{\xi}T^{3/2}\sum_n\frac{\ell(n)}{n^c}
\frac{1}{|n-T_\circ|+T^{1/2}}.
\ee
Modifying slightly the proof of \cite[Thm.~3.1]{Banks3}, we find that
the sum in \eqref{eq:I1 E2bd} is $\ll T^{-1}(\log T)^3$. Therefore,
using \eqref{eq:I1 E1bd} and \eqref{eq:I1 E2bd} along with \eqref{eq:gosh1},
we see that
\[
I_1=\Sigma_2(T)+O_{ \gchm ,\xi}(T^{1/2+\eps}).
\]

Next, by \eqref{eq:horiz-up} and Lemma~\ref{lem:Xxexpansion}, we have
\[
D(\sigma+it,\gch)\llsym{ \gchm }\tau^\eps
\mand
\cM_{\opch}(1-\sigma-it)\llsym{\pchm}\tau^{\sigma-1/2}
\qquad(\sigma\in[b,c],~t\ge 1),
\]
from which we derive that
\[
I_2=\frac{1}{2\pi i}\int_{c+iT}^{b+iT}D(s,\gch)\,\xi^{-s}\cM_{\opch}(1-s)\,ds
\llsym{ \gchm ,\xi}T^{1/2+\eps}.
\]
Similarly, we have
\[
I_3\llsym{ \gchm ,\xi}T^\eps
\mand
I_4\llsym{ \gchm ,\xi}1.
\]

Combining \eqref{eq:Ibegin} with the above estimates for the integrals $I_j$,
we obtain \eqref{eq:Sigsum}, finishing the proof.
\end{proof}

\begin{corollary}\label{cor:vonMangoldt-twist2}
Let $\gch$ be a character mod~$ \gchm $ induced from a primitive character $\pch$
mod~$\pchm$ satisfying $\LHpch$. For any $\xi\in\R^+$, $\euB\in C_c^\infty(\R^+)$, and
$X\ge 10$, 
\dalign{
&\ssum{\rho=\beta+i\gamma\\\beta\ge\frac12,\,\gamma>0}\xi^{-\rho}L'(\rho,\gch)
\cM_{\opch}(1-\rho)\euB\Big(\frac{\gamma}{2\pi\xi X}\Big)\\
&\qquad\qquad-\frac{\tau(\opch)}{\pchm}
\sum_n\ell(n)\gch(n)\e(-n\xi/\pchm)\euB(n/(\pchm X))\llsym{ \gchm ,\xi,\euB}X^{1/2+\eps}.
}
\end{corollary}

\begin{proof}
As in the proof of Theorem~\ref{thm:vonMangoldt-twist}, we define for $u>0$:
\[
\Sigma_1(u)\defeq\ssum{\rho=\beta+i\gamma\\\beta\ge\frac12\\0<\gamma\le u}
\xi^{-\rho}L'(\rho,\gch)\cM_{\opch}(1-\rho),\qquad
\Sigma_2(u)\defeq\frac{\tau(\opch)}{\pchm}
\hskip-5pt\ssum{n\le \pchm u/(2\pi\xi)}
\hskip-5pt\ell(n)\gch(n)\e(-n\xi).
\]
By Theorem~\ref{thm:vonMangoldt-twist}
with $T\defeq 2\pi\xi Xu$, we have
\be\label{eq:lawnmower}
\Sigma_1(2\pi\xi Xu)-\Sigma_2(2\pi\xi Xu)
\llsym{ \gchm ,\xi}(Xu)^{1/2+\eps}+1\qquad(u>0).
\ee
Next, we denote
\dalign{
\Sigma_3(u)&\defeq\Sigma_2(2\pi\xi u/\pchm)
=\frac{\tau(\opch)}{ \gchm }\ssum{n\le u}\ell(n)\gch(n)\e(-n\xi),\\
\Sigma_4(X)&\defeq\ssum{\rho=\beta+i\gamma\\\beta\ge\frac12,\,\gamma>0}\xi^{-\rho}L'(\rho,\gch)
\cM_{\opch}(1-\rho)\euB\Big(\frac{\gamma}{2\pi\xi X}\Big),\\
\Sigma_5(X)&\defeq\frac{\tau(\opch)}{\pchm}
\sum_n\ell(n)\gch(n)\e(-n\xi/\pchm)\euB(n/(\pchm X)).
}
Using Riemann-Stieltjes integration, we have
\dalign{
\Sigma_4(X)&=\int_0^\infty\euB\Big(\frac{u}{2\pi\xi X}\Big)\,d\Sigma_1(u)
=\int_0^\infty\euB(u)\,d\Sigma_1(2\pi\xi Xu)\\
&=-\int_0^\infty\euB'(u)\Sigma_1(2\pi\xi Xu)\,du,
}
and
\dalign{
\Sigma_5(X)&=\int_0^\infty\euB(u/(\pchm X))\,d\Sigma_3(u)
=\int_0^\infty\euB(u)\,d\Sigma_3(\pchm Xu)\\
&=-\int_0^\infty\euB'(u)\Sigma_3(\pchm Xu)\,du
=-\int_0^\infty\euB'(u)\Sigma_2(2\pi\xi Xu)\,du.
}
Hence, using \eqref{eq:lawnmower}, we get that
\[
\Sigma_4(X)-\Sigma_5(X)\llsym{ \gchm ,\xi}\int_0^\infty\big|\euB'(u)\big|
\big((Xu)^{1/2+\eps}+1\big)\,du\llsym{\euB}X^{1/2+\eps},
\]
which finishes the proof.
\end{proof}

\section{Computation under $\RHsimstar$ and $\LHstar$}

\begin{lemma}\label{lem:aloevera}
Assume $\RHsimstar$ and $\LHstar$.
Let $\pch$ be a primitive character mod~$\pchm$. Let $\xi=h/k$
be a rational number with $h,k>0$ and $(h,k)=1$. Put $ \gchm \defeq\pchm k$,
and let $\chi\defeq\pch\cdot\chi_{0, \gchm }$ be the character mod~$ \gchm $ 
induced from $\pch$. For all $\euB\in C_c^\infty(\R^+)$ and $X\ge 10$, we have
\be
\label{eq:meowing}
\frac{\tau(\opch)}{\pchm}
\sum_{n}\ell(n)\gch(n)\e(-n\xi/\pchm)\euB(n/(\pchm X))
=C_{\pch,\xi}\cdot F_{\pch,\xi}(X)+O_{ \gchm ,\xi,\euB}(X^{1/2+\eps}),
\ee
where $ \gchm \defeq \pchm k$, $\gch$ is the character mod~$ \gchm $ induced from $\pch$,
$C_{\pch,\xi}$ is given by \eqref{eq:CXxi-defn},
and $F_{\pch,\xi}(X)$ is given by \eqref{eq:FXxiX-defn}.
\end{lemma}

\begin{proof} 
The character $\gch$ is supported on integers
coprime to $ \gchm $, hence the sum in \eqref{eq:meowing} is equal to
\dalign{
&\ssum{a\bmod  \gchm \\(a, \gchm )=1}\e(-ah/ \gchm )\pch(a)
\sum_{n\equiv a\bmod  \gchm }\ell(n)\euB(n/(\pchm X))\\
&\qquad=\frac{1}{\phi( \gchm )}\ssum{a\bmod  \gchm \\(a, \gchm )=1}\e(-ah/ \gchm )\pch(a)
\sum_{\gch'\bmod  \gchm }\overline{\gch'}(a)
\sum_n\ell(n)\gch'(n)\euB(n/(\pchm X)),
}
where the middle sum runs over all characters $\gch'$ mod $ \gchm $.
By Theorem~\ref{thm:ultraclean}\,$(iii)$
the total contribution from all nonprincipal characters $\gch'$
is $O_{\gch,\xi,\euB}(X^{1/2+\eps})$.
For the principal character $\gch'=\gch_{0, \gchm }$,
the contribution is
\be\label{eq:9oclock}
\frac{c}{\phi( \gchm )}\int_0^\infty\euB(u/(\pchm X))P'_{\gchm} (u)\,du
+O_{\gch,\xi,\euB}(X^{1/2+\eps}),
\ee
where we have used Theorem~\ref{thm:ultraclean}\,$(iii)$ again, and
\[
c\defeq\ssum{a\bmod  \gchm \\(a, \gchm )=1}\e(-ah/ \gchm )\pch(a);
\]
note that the integral in \eqref{eq:9oclock} is $F_{\pch,\xi}(X)$ by definition.
By \cite[Theorem~9.12]{MontVau}, we have
\[
c=\begin{cases}
\overline\pch(-h)\pch(k)\mu(k)\tau(\pch)
&\quad\hbox{if $(h,\pchm)=1$},\\
0&\quad\hbox{otherwise}.\\
\end{cases}
\]
Using the well known relation
\be\label{eq:tautauq}
\tau(\pch)\tau(\opch)=\pch(-1)\pchm
\ee
for the Gauss sums defined in \eqref{eq:red}, and
the fact that $(h, \gchm )=(h,\pchm)$, we obtain the stated result.
\end{proof}

\section{Proof of Theorem~\ref{thm:main}}

Throughout the proof, $\LHstar$ is assumed to hold.
Once and for all, let $\pch$ be a fixed primitive character mod~$\pchm$.

In one direction, suppose that $\RHsimstar$ is true. In particular, $\RHsimpch$ holds,
and the first condition of $\RHdagsimpch$ is verified. Let
$\xi\defeq h/k$ with $h,k>0$ and $(h,k)=1$,
$\euB\in C_c^\infty(\R^+)$, and  $X\ge 10$.
As $\RHsimpch$ holds, Corollary~\ref{cor:vonMangoldt-twist2} gives
\[
\begin{split}
&\ssum{\rho=\frac12+i\gamma\\\gamma>0}\xi^{-\rho}L'(\rho,\gch)
\cM_{\opch}(1-\rho)\euB\Big(\frac{\gamma}{2\pi\xi X}\Big)\\
&\qquad\qquad-\frac{\tau(\opch)}{\pchm}
\sum_n\ell(n)\gch(n)\e(-n\xi/\pchm)\euB(n/(\pchm X))\llsym{\pchm,\xi,\euB}X^{1/2+\eps}.
\end{split}
\]
Using estimate \eqref{eq:meowing} of Lemma~\ref{lem:aloevera}, we
obtain \eqref{eq:eureka}. Since $\xi$ is arbitrary, the second
condition of $\RHdagsimpch$ is verified. Thus, $\RHdagsimpch$ is true,
and the proof is complete in this direction.

In the other direction, suppose that $\RHdagsimpch$ is true.
To prove the theorem, we show that $\RHsimpcho$ holds for an arbitrary
primitive character $\pcho$ mod~$\pchom$.

Observe that if $\gcho$ is any character induced from $\pcho$,
then $\RHsimgcho$ and $\RHsimpcho$ are equivalent since
$L(s,\gcho)$ and $L(s,\pcho)$ have the same zeros in the critical strip.
Therefore, it suffices to show that $\RHsimgcho$ holds for \emph{some}
character~$\gcho$ induced from $\pcho$.
For this purpose, we define
\[
\gcho\defeq\pcho\cdot\gch_{0,\pchm\pchom},\qquad
\gch\defeq\pch\cdot\gch_{0,\pchm\pchom},\qquad
\vartheta\defeq\ogch\cdot\gcho,
\]
and turn our attention to the sum
\be\label{eq:W1}
\cW\defeq\sum_n\ell(n)\gcho(n)\euB(n/(\pchm X))
=\sum_{(n,\pchm\pchom)=1}\ell(n)\vartheta(n)\gch(n)\euB(n/(\pchm X)).
\ee
If $(n,\pchm\pchom)=1$, then (cf.\ \cite[Theorem~9.5]{MontVau})
\[
\vartheta(n)\tau(\otheta)
=\sum_{h\bmod \pchm\pchom}\otheta(h)\e(hn/(\pchm\pchom))
=\ssum{1\le h\le \pchm\pchom\\(h,\pchm\pchom)=1}
\otheta(-h)\e(-hn/(\pchm\pchom)),
\]
and it follows that
\[
\cW=\frac{1}{\tau(\otheta)}\ssum{1\le h\le \pchm\pchom\\(h,\pchm\pchom)=1}
\otheta(-h)\sum_{(n,\pchm\pchom)=1}\ell(n)
\e(-hn/(\pchm\pchom))\chi(n)\euB(n/(\pchm X)).
\]
Applying Corollary~\ref{cor:vonMangoldt-twist2} and then
\eqref{eq:eureka} to each inner sum, we have (with $\xi\defeq h/k$)
\dalign{
&\sum_{(n,\pchm\pchom)=1}\ell(n)\e(-hn/(\pchm\pchom))\chi(n)\euB(n/(\pchm X))\\
&\qquad=\frac{\pchm}{\tau(\opch)}
\ssum{\rho=\beta+i\gamma\\\beta\ge\frac12,\,\gamma>0}\xi^{-\rho}L'(\rho,\gch)
\cM_{\opch}(1-\rho)\euB\Big(\frac{\gamma}{2\pi \xi X}\Big)
+O_{\pchm,k,\euB}(X^{1/2+\eps})\\
&\qquad=\frac{\pchm}{\tau(\opch)}\,
C_{\pch,h/k}\cdot F_{\pchm,k}(X)+O_{\pchm,k,\euB}(X^{1/2+\eps});
}
thus,
\[
\cW=\frac{\pchm}{\tau(\otheta)\tau(\opch)}\,F_{\pchm,k}(X)
\ssum{1\le h\le \pchm\pchom\\(h,\pchm\pchom)=1}
\otheta(-h)C_{\pch,h/k} +O_{\pchm,k,\euB}(X^{1/2+\eps}).
\]
Finally, if $(h,\pchm\pchom)=1$, then (see \eqref{eq:CXxi-defn})
\[
\otheta(-h)\cdot C_{\pch,h/k}
=\pch(-h)\ogcho(-h)\cdot
\frac{\opch(h)\pch(k)\mu(k)}{\phi(\pchm k)}
=\ogcho(-h)\cdot\frac{\pch(-k)\mu(k)}{\phi(\pchm k)},
\]
and so
\be\label{eq:W2}
\cW=\frac{\pchm\,\pch(-k)\mu(k)}{\tau(\otheta)\tau(\opch)\phi(\pchm k)}\,F_{\pchm,k}(X)
\ssum{1\le h\le \pchm\pchom\\(h,\pchm\pchom)=1}\ogcho(-h)+O_{\pchm,k,\euB}(X^{1/2+\eps}).
\ee

We are now in a position to complete the proof. Using orthogonality,
we evaluate the sum in \eqref{eq:W2} as follows:
\be\label{eq:W3}
\ssum{1\le h\le \pchm\pchom\\(h,\pchm\pchom)=1}\ogcho(-h)
=\begin{cases}
\phi(\pchm\pchom)&\quad\hbox{if $\pcho=\one$,}\\
0&\quad\hbox{if $\pcho\ne\one$.}\\
\end{cases}
\ee
In the case that $\gcho$ is nonprincipal, we have $\pcho\ne\one$. Hence,
combining \eqref{eq:W1}$-$\eqref{eq:W3}, we derive the bound
\[
\sum_n\ell(n)\gcho(n)\euB(n/(\pchm X))\llsym{\pchm,k,\euB}X^{1/2+\eps}.
\]
The dependence on $\pchm$ can be ignored since the character $\pch$ mod~$\pchm$
is \emph{fixed}. Taking into account that $\delta_{\gcho}=0$,
the estimate \eqref{eq:bird2} of Theorem~\ref{thm:ultraclean} is verified, and
thus $\RHsimgcho$ is true. On the other hand, if $\gcho$ is principal, then
\[
\gcho=\gch_{0,q},\qquad
\delta_{\gcho}=1,\qquad
\pcho=\one,\qquad
k=1,\qquad
\gch\defeq\pch,\qquad
\vartheta\defeq\opch.
\]
Combining \eqref{eq:tautauq} and \eqref{eq:W1}$-$\eqref{eq:W3}, we get that
\dalign{
\sum_n\ell(n)\gcho(n)\euB(n/(\pchm X))&=F_{\pch,1}(X)+O_{\pchm,\euB}(X^{1/2+\eps})\\
&=\delta_{\gcho}
\int_0^\infty\euB(u/(\pchm X))P'_{\pchm}(u)\,du+O_{\pchm,\euB}(X^{1/2+\eps}).
}
Replacing $\pchm X$ by $X$, we again verify
estimate \eqref{eq:bird2} of Theorem~\ref{thm:ultraclean},
hence $\RHsimgcho$ is true in this case as well.

\end{document}